\documentclass[pdftex]{article} 
\usepackage{amsmath,amssymb,amsthm,color} 

\title{Interpolation properties for the bimodal provability logic $\mathbf{GR}$}
\author{Haruka Kogure\footnote{Email: kogure1987@stu.kanazawa-u.ac.jp}
\footnote{College of Science and Engineering, School of Mathematics and Physics, Kanazawa University, Kakuma, Kanazawa 920-1192, Japan}
and Taishi Kurahashi\footnote{Email: kurahashi@people.kobe-u.ac.jp}
\footnote{Graduate School of System Informatics, Kobe University, 1-1 Rokkodai, Nada, Kobe 657-8501, Japan.}}
\date{}

\theoremstyle{plain}
\newtheorem{thm}{Theorem}[section]
\newtheorem*{thm*}{Theorem}
\newtheorem{lem}[thm]{Lemma}
\newtheorem{prop}[thm]{Proposition}
\newtheorem{cor}[thm]{Corollary}
\newtheorem{fact}[thm]{Fact}
\newtheorem*{fact*}{Fact}
\newtheorem{prob}[thm]{Problem}
\newtheorem*{prob*}{Problem}
\newtheorem*{cl*}{Claim}
\newtheorem*{scl*}{Subclaim}

\theoremstyle{definition}
\newtheorem{defn}[thm]{Definition}

\newcommand{\PA}{\mathsf{PA}}
\newcommand{\PR}{\mathrm{Pr}}
\newcommand{\PRR}{\mathrm{Pr}^{\mathrm{R}}}

\newcommand{\seq}[1]
{\langle#1\rangle}

\newcommand{\N}{\mathbf{N}}
\newcommand{\NR}{\mathbf{NR}}
\newcommand{\GL}{\mathbf{GL}}

\newcommand{\GR}{\mathbf{GR}}
\newcommand{\GRM}{\mathbf{GR}^-}
\newcommand{\GRC}{\mathbf{GR}^\circ}
\newcommand{\GRB}{\mathbf{GR}^\bullet}

\newcommand{\bb}{\blacksquare}

\newcommand{\bg}{\textcolor[gray]{0.7}{\blacksquare}}
\newcommand{\LT}{\mathcal{L}_2}
\newcommand{\LBox}{\mathcal{L}_\Box}
\newcommand{\Lbb}{\mathcal{L}_\bb}

\newcommand{\MF}{\mathrm{MF}}
\newcommand{\MFT}{\mathrm{MF}_2}

\newcommand{\acc}{\sqsubset}

\begin{document}

\maketitle

\begin{abstract}
    We study interpolation properties for Shavrukov's bimodal logic $\GR$ of usual and Rosser provability predicates. 
    For this purpose, we introduce a new sublogic $\GRC$ of $\GR$ and its relational semantics. 
    Based on our new semantics, we prove that $\GR^\circ$ and $\GR$ enjoy Lyndon interpolation property and uniform interpolation property. 
\end{abstract}

\section{Introduction}

The modal logic $\GL$ is known as the logic of provability. 
Let $\PR_{\PA}(x)$ be a natural provability predicate of $\PA$. 
Well-known Solovay's arithmetical completeness theorem \cite{Solovay} states that the theorems of $\GL$ are exactly the $\PA$-verifiable modal principles under the interpretation of $\Box$ as $\PR_{\PA}$. 
It has been investigated that the logic $\GL$ has good properties such as the Craig interpolation property (CIP), which was independently proved by Smory\'nski \cite{Smorynski} and Boolos \cite{Boolos}. 
In particular, CIP for $\GL$ is important because Smory\'nski showed that the de Jongh--Sambin fixed point theorem for $\GL$ follows from CIP. 

Various non-standard provability predicates have been studied so far, and the most important of which are Rosser provability predicates $\PRR_{\PA}(x)$.
Shavrukov \cite{Shavrukov} introduced the bimodal logic $\GR$ having axioms concerning two modal operators $\Box$ and $\bb$ that are intended to be respectively interpreted as $\PR_{\PA}(x)$ and $\PRR_{\PA}(x)$, and proved that $\GR$ is arithmetically complete with respect to this interpretation. 
For the proof of the arithmetical completeness theorem, Shavrukov introduced an auxiliary Kripke semantics for $\GR$ that deals with formulas of the form $\bb B$ as atomic formulas, and proved the Kripke completeness of $\GR$ with respect to this semantics.
Using this semantics and CIP for $\GL$, Sidon \cite{Sidon} proved that $\GR$ enjoys CIP.

Uniform interpolation property (UIP) and Lyndon interpolation property (LIP) are know as stronger notions than CIP. 
UIP for $\GL$ was proved by Shavrukov \cite{Shavrukov93}. 
The question of whether the logic $\GL$ possesses LIP was proposed by Maksimova \cite[p.~469]{Maksimova} (see also Artemov and Beklemishev \cite[p.~208]{AB}) and was subsequently affirmatively answered by Shamkanov \cite{Shamkanov}. 

The purpose of the present paper is to analyze LIP and UIP for $\GR$ as a continuation of Sidon's work.
However, there are some obstacles when considering LIP in particular. 
Since LIP deals with the positive and negative occurrences of propositional variables in formulas, Sidon's strategy of proving LIP for $\GR$ from LIP for $\GL$ does not seem to work. 
So, we will adopt a strategy of proving LIP for $\GR$ directly semantically following Shamkanov's method of the proof of LIP for $\GL$. 
However, even in this case, since Shavrukov's semantics is only an auxiliary one, it does not seem to be available for the proof of LIP for $\GR$.
Here we focus on the logic $\NR$ which is recently introduced by the second author as the logic of all Rosser provability predicates. 
It is shown in \cite{Kurahashi} that $\NR$ exactly axiomatizes the $\bb$-fragment of $\GR$, and furthermore, $\NR$ is sound and complete with respect to a Kripke-like semantics introduced by Fitting, Marek and Truszczy{\'n}ski \cite{fmt}.
Therefore, we adopt a strategy of proving LIP for $\GR$ by introducing a new semantics obtained by combining the usual Kripke semantics and the semantics for $\NR$.

In Section \ref{sec_pre}, we introduce some basic notions and overview several already known facts. 
In Section \ref{sec_GRC}, we introduce a new sublogic $\GRC$ of $\GR$ and a new relational semantics for $\GRC$ and $\GR$ as mentioned above. 
Our logic $\GRC$ is the base logic of our new semantics, just as Shavrukov introduced the sublogic $\GRM$ of $\GR$ to his semantics. 
In particular, $\GRC$ is intermediate between the logics $\GRM$ and $\GR$. 
Section \ref{sec_LIP_GRC} is devoted to proving LIP for $\GRC$. 
Our proof is based on the newly introduced semantics, and then as a consequence of our proof, we obtain the completeness and finite frame property of $\GRC$ with respect to that semantics. 
In Section \ref{sec_LIP_GR}, we prove that $\GR$ also enjoys LIP. 
As a consequence, we obtain the completeness and finite frame property of $\GR$ with respect to our semantics. 
Moreover, we show that $\GRC + \dfrac{\neg B}{\neg \bb B}$ is an alternative axiomatization of $\GR$. 
In Section \ref{sec_UIP}, through the investigation of the relationships between $\GRM$, $\GRC$, and $\GR$, we prove UIP for $\GRM$, $\GRC$, and $\GR$ based on Sidon's method.
Finally in Section \ref{sec_fw}, we discuss some questions that remain unanswered. 

\section{Preliminaries}\label{sec_pre}

Let $\LT$ denote the language of bimodal propositional logic having countably many propositional variables, propositional constant $\bot$, propositional connectives $\neg$ and $\vee$, and two unary modal operators $\Box$ and $\bb$. 
Other connectives such as $\wedge$ and $\to$ are defined by using $\neg$ and $\vee$ in the usual way. Also, the modal operator $\Diamond$ is the abbreviation of $\neg \Box \neg$. 
Let $\LBox$ and $\Lbb$ be $\Box$-fragment and $\bb$-fragment of $\LT$, respectively. 

\begin{defn}
The logic $\GRM$ in the language $\LT$ is defined as follows: 
    \begin{itemize}
        \item The axioms of $\GRM$ are:
        \begin{enumerate}
            \item All propositional tautologies in the language $\LT$. 
            \item $\Box (A \to B) \to (\Box A \to \Box B)$. 
            \item $\Box (\Box A \to A) \to \Box A$.  
            \item $\bb A \to \Box A$. 
            \item $\Box A \to \Box \bb A$. 
            \item $\Box A \to (\Box \bot \vee \bb A)$.
            \item $\Diamond \bb A \to \Diamond A$. 
            \end{enumerate}
        \item The inference rules of $\GRM$ are modus ponens $\dfrac{A \quad A \to B}{B}$ and $\Box$-necessitation $\dfrac{A}{\Box A}$. 
    \end{itemize}
The logic $\GR$ is obtained from $\GRM$ by adding the rule $\dfrac{\Box A}{A}$. 
\end{defn}

From the arithmetical soundness of $\GR$ (\cite[Lemma 2.5]{Shavrukov}) and Solovay's arithmetical completeness theorem of $\GL$, the following conservation result is obtained. 
(Of course, this conservation result can be proved by using Kripke models.) 

\begin{fact}\label{conservation_GR}
$\GR$ is a conservative extension of $\GL$. 
That is, for any $\LBox$-formula $A$, if $\GR \vdash A$, then $\GL \vdash A$. 
\end{fact}

 Shavrukov introduced an auxiliary Kripke semantics for $\GRM$ and $\GR$ to prove the arithmetical completeness theorem for $\GR$. 
Let $\MF_\bb$ be the set of all $\Lbb$-formulas.

\begin{defn}  \leavevmode
\begin{itemize}
    \item We say that a quadruple $(W, \acc, r, \Vdash)$ is a \textit{$\GR^-$-model} iff $W$ is a non-empty set, $\acc$ is a transitive and conversely well-founded binary relation on $W$, $r$ is the root element of $W$ with respect to $\acc$, and $\Vdash$ is a satisfaction relation on $W \times \MF_\bb$ fulfilling the usual conditions for propositional connectives and the following conditions: for $w \in W$, 
    \begin{itemize}
        \item $w \Vdash \Box B \iff \forall x \in W\, (w \acc x \Rightarrow x \Vdash B)$,
        \item $w \Vdash B$ for each axiom $B$ of $\GRM$. 
    \end{itemize}
    \item A $\GRM$-model $(W, \acc, r, \Vdash)$ is called \textit{non-trivial} iff there exists an element $w \in W$ such that $r \neq w$. 
\end{itemize}
\end{defn}

By using the Kripke completeness of $\GL$, Shavrukov proved the soundness and completeness of $\GRM$ and $\GR$ with respect to the corresponding classes of $\GRM$-models. 

\begin{fact}[Shavrukov {\cite[Theorems 1.8 and 1.9]{Shavrukov}}]\label{compl_GRM}
Let $A$ be any $\LT$-formula. 
\begin{enumerate}
    \item $\GRM \vdash A$ if and only if $A$ is valid in all $\GRM$-models. 
    \item $\GR \vdash A$ if and only if $A$ is valid in all non-trivial $\GRM$-models. 
\end{enumerate}
\end{fact}

As an application of Fact \ref{compl_GRM}, the following useful fact is obtained. 

\begin{fact}[Shavrukov {\cite[Corollary 1.10]{Shavrukov}}]\label{GRM_GR}
For any $\LT$-formula $A$, $\GR \vdash A$ if and only if $\GRM \vdash \Box A$. 
\end{fact}

Sidon \cite{Sidon} applied Fact \ref{compl_GRM} to the study of interpolation property of $\GRM$ and $\GR$. 
For each formula $A$, let $v(A)$ denote the set of all propositional variables occurring in $A$. 

\begin{defn}[CIP]
We say that a logic $L$ enjoys \textit{Craig interpolation property} (\textit{CIP}) iff for any formulas $A$ and $B$, if $L \vdash A \to B$, then there exists a formula $C$ such that $v(C) \subseteq v(A) \cap v(B)$, $L \vdash A \to C$, and $L \vdash C \to B$. 
\end{defn}

\begin{fact}[Sidon {\cite[Theorems 2.1 and 2.3]{Sidon}}]
The logics $\GRM$ and $\GR$ enjoy Craig interpolation property. 
\end{fact}


Shavrukov introduced $\GR$ as the bimodal provability logic of the usual provability predicate and all Rosser provability predicates. 
Recently, the second author introduced the logic $\NR$ and proved that this logic is the provability logic of all Rosser provability predicates \cite{Kurahashi}. 
The logic $\NR$ is a non-normal modal logic based on Fitting--Marek--Truszczy{\'n}ski's pure logic of necessitation $\N$ \cite{fmt}. 

\begin{defn}
The logic $\N$ in the language $\Lbb$ is defined as follows: 
    \begin{itemize}
        \item The axioms of $\N$ are all propositional tautologies in the language $\Lbb$.
        \item The inference rules of $\N$ are modus ponens and $\bb$-necessitation $\dfrac{A}{\bb A}$. 
    \end{itemize}
The logic $\NR$ is obtained from $\N$ by adding the \textit{Rosser rule} $\dfrac{\neg A}{\neg \bb A}$. 
\end{defn}

It is proved that $\NR$ axiomatizes exactly the $\bb$-fragment of $\GR$. 

\begin{fact}[Kurahashi {\cite[Thorem 6.2]{Kurahashi}}]
$\GR$ is a conservative extension of $\NR$. 
That is, for any $\Lbb$-formula $A$, $\NR \vdash A$ if and only if $\GR \vdash A$.
\end{fact}

Fitting, Marek and Truszczy{\'n}ski introduced a Kripke-like relational semantics for $\N$. 

\begin{defn} \leavevmode
  \begin{itemize}
    \item A tuple $(W, \{\prec_B\}_{B \in \MF_{\bb}})$ is called an \textit{$\N$-frame} iff $W$ is a non-empty set and $\prec_B$ is a binary relation on $W$ for every $B \in \MF_{\bb}$. 
    \item A triple $(W, \{\prec_B\}_{B \in \MF_{\bb}}, \Vdash)$ is said to be an \textit{$\N$-model} iff $(W, \{\prec_B\}_{B \in \MF_{\bb}})$ is an $\N$-frame and $\Vdash$ is a satisfaction relation between $W$ and $\MF_\bb$ fulfilling the usual conditions for connectives and the following condition:
    \[
        w \Vdash \bb B \iff \forall x \in W\, (w \prec_B x \Rightarrow x \Vdash B).
    \]
    We say an $\N$-frame $(W, \{\prec_B\}_{B \in \MF_{\bb}})$ is \textit{serial} iff for all $B \in \MF_\bb$ and $w \in W$, there exists an $x \in W$ such that $w \prec_B x$. 
    \end{itemize}
\end{defn}

It is shown that the logics $\N$ and $\NR$ are sound and complete with respect to the above semantics:

\begin{fact}[Fitting, Marek and Truszczy{\'n}ski {\cite[Theorems 3.6 and 4.10]{fmt}}] \label{ffp_N}
For any $A \in \MF_\bb$, $\N \vdash A$ if and only if $A$ is valid on all $\N$-frames.
\end{fact}

\begin{fact}[Kurahashi {\cite[Theorem 3.12]{Kurahashi}}] \label{ffp_NR}
For any $A \in \MF_\bb$, $\NR \vdash A$ if and only if $A$ is valid on all serial $\N$-frames.
\end{fact}

\section{The logic $\GRC$ and its semantics}\label{sec_GRC}

In this section, we introduce the new sublogic $\GRC$ of $\GR$ and its relational semantics. 
Our semantics is a combination of usual Kripke semantics and Fitting--Marek--Truszczy{\'n}ski's semantics. 

\begin{defn}
    The logic $\GRC$ is obtained from $\GRM$ by adding $\bb$-necessitation $\dfrac{A}{\bb A}$. 
\end{defn}

Since the rule $\dfrac{A}{\bb A}$ is admissible in $\GR$, we have that $\GRC$ is intermediate between $\GRM$ and $\GR$. 
So the following proposition immediately follows from Fact \ref{GRM_GR}. 

\begin{prop}\label{GRC_GR}
For any $\LT$-formula $A$, $\GR \vdash A$ if and only if $\GRC \vdash \Box A$. 
\end{prop}

We introduce our relational semantics for $\GRC$. 
Let $\MFT$ be the set of all $\LT$-formulas. 

\begin{defn}[$\GRC$-frames and $\GRC$-models] \leavevmode
  \begin{itemize}
    \item A triple $(W, \acc, \{\prec_B\}_{B \in \MFT})$ is called a \textit{$\GRC$-frame} iff $W$ is a non-empty set, $\acc$ is a transitive and conversely well-founded binary relation on $W$ and for every $B \in \MFT$, $\prec_B$ is a binary relation on $W$ satisfying the following four conditions: 
    \begin{itemize}
        \item $\forall x, y \in W\, (x \acc y \Rightarrow x \prec_B y)$, 
        \item $\forall x, y, z \in W\, (x \acc y\ \&\ y \prec_B z \Rightarrow x \acc z)$, 
        \item $\forall x, y, z \in W\, (x \prec_B y\ \&\ x \acc z \Rightarrow x \acc y)$, 
        \item $\forall x, y \in W\, \bigl(x \acc y \Rightarrow \exists z \in W\, (y \prec_B z\ \&\ x \acc z) \bigr)$. 
    \end{itemize}

   \item A quadruple $(W, \acc, \{\prec_B\}_{B \in \MFT}, \Vdash)$ is said to be a \textit{$\GRC$-model} iff $(W, \acc, \{\prec_B\}_{B \in \MFT})$ is a $\GRC$-frame and $\Vdash$ is a satisfaction relation on $W \times \MFT$ fulfilling the usual conditions for each propositional connective and the following two conditions: 
    \begin{itemize}
        \item $w \Vdash \Box B \iff \forall x \in W\, (w \acc x \Rightarrow x \Vdash B)$,
        \item $w \Vdash \bb B \iff \forall x \in W\, (w \prec_B x \Rightarrow x \Vdash B)$.   
    \end{itemize}
    \end{itemize}
\end{defn}

The following proposition saying that $\GRC$ is actually sound with respect to our $\GRC$-frames is easily proved by induction on the length of proofs of $A$ in $\GRC$ and so we leave the proof to the reader. 

\begin{prop}\label{soundness_GRC}
For any $\LT$-formula $A$, if $\GRC \vdash A$, then $A$ is valid on all $\GRC$-frames. 
\end{prop}




The following proposition says that the logic $\GRC$ is strictly intermediate between the logics $\GRM$ and $\GR$. 

\begin{prop}
\begin{enumerate}
    \item $\GRC \vdash \bb \neg \bot$ and $\GR \vdash \neg \bb \bot$, 
    \item $\GRM \nvdash \bb B$ for any $\LT$-formula $B$,   
    \item $\GRC \nvdash \neg \bb B$ for any $\LT$-formula $B$. 
\end{enumerate}
\end{prop}
\begin{proof}
The first clause is easily proved. 

2. Consider a model $\mathcal{M} = (W, \acc, r, \Vdash)$, where $W = \{r\}$, $\acc = \emptyset$, and $r \nVdash \bb B$ for all $B \in \MFT$. 
It is easy to see that $\mathcal{M}$ is a $\GRM$-model. 
So, by Fact \ref{compl_GRM}, we obtain $\GRM \nvdash \bb B$. 

3. Consider a model $\mathcal{M} = (W, \acc, \{\prec_B\}_{B \in \MFT}, \Vdash)$, where $W = \{r\}$, $\acc = \prec_B = \emptyset$ for all $B \in \MFT$. 
It can be seen that $\mathcal{M}$ is a $\GRC$-model. 
Since $r \Vdash \bb B$, by Proposition \ref{soundness_GRC}, we have $\GRC \nvdash \neg \bb B$. 
\end{proof}

We close this section with the following conservation theorem. 

\begin{thm}\label{conservation}
$\GRC$ is a conservative extension of $\N$. 
That is, for any $\Lbb$-formula $A$, if $\GRC \vdash A$, then $\N \vdash A$. 
\end{thm}
\begin{proof}
We prove the contrapositive. 
Suppose $\N \nvdash A$, then by Fact \ref{ffp_N}, we find an $\N$-model $(W, \{\prec_B\}_{B \in \MF_\bb}, \Vdash)$ and $w \in W$ such that $w \nVdash A$. 
We expand this $\N$-model to a $\GRC$-model $(W, \acc, \{\prec_B\}_{B \in \MFT}, \Vdash^*)$ by letting $\acc = \prec_B = \emptyset$ for all $B \in \MFT \setminus \MF_\bb$. 
It is easy to see that $w \nVdash^* A$, and hence we obtain $\GRC \nvdash A$ by Proposition \ref{soundness_GRC}. 
\end{proof}

\section{Lyndon interpolation property for $\GRC$}\label{sec_LIP_GRC}

In this section, we prove that $\GRC$ enjoys Lyndon interpolation property. 
As a consequence of our proof, we also prove that $\GRC$ is complete and has the finite frame property with respect to $\GRC$-frames.

For every $\LT$-formula $A$, the set $S^+(A)$ of all positive subformulas and the set $S^-(A)$ of all negative subformulas of $A$ are recursively defined as follows:
\begin{itemize}
\item $S^+(A) = \{A\}$ when $A$ is either a propositional variable or $\bot$, 
\item $S^+(B \vee C) = S^+(B) \cup S^+(C) \cup \{B \vee C\}$ and $S^-(B \vee C) = S^-(B) \cup S^-(C)$, 
\item $S^+(\neg B) = S^-(B) \cup \{\neg B\}$ and $S^-(\neg B) = S^+(B)$, 
\item $S^+ (\bg B) = S^+ (B) \cup \{\bg B\}$ and $S^-(\bg B) = S^-(B)$ for $\bg \in \{\Box, \bb\}$.
\end{itemize}

For any $\LT$-formula $A$, let $v^\ast(A) : = S^\ast(A) \cap v(A)$ for $\ast \in \{+, -\}$ and $S(A) := S^+(A) \cup S^-(A)$. 
Here, $S(A)$ is the set of all subformulas of $A$. 

\begin{defn}
We say that a logic $L$ enjoys \textit{Lyndon interpolation property} (\textit{LIP}) iff for any formulas $A$ and $B$, if $L \vdash A \to B$, then there exists a formula $C$ such that $v^*(C) \subseteq v^\ast(A) \cap v^\ast(B)$ for $\ast \in \{+, -\}$, $L \vdash A \to C$, and $L \vdash C \to B$.
We say that such a formula $C$ is a \textit{Lyndon interpolant} of $A \to B$ in $L$. 
\end{defn}

For each $\LT$-formula $A$, we define the set $\tau(A)$ as follows: 
\[
    \tau(A) := \{ p \mid p\in v^+ (A)\} \cup \{ \neg p \mid p \in v^- (A)\}. 
\]
The set $\tau(A)$ is introduced to simplify the notation, and the idea of using this set is due to Shamkanov \cite{Shamkanov}. 
More precisely, it is shown  that `$v^* (C) \subseteq v^* (A) \cap v^* (B)$ for every $* \in \{+, -\}$' if and only if $\tau(C) \subseteq \tau(A) \cap \tau(B)$. 
So, Lyndon interpolant of $A \to B$ in $L$ is exactly a formula $C$ satisfying $\tau(C) \subseteq \tau(A) \cap \tau(B)$, $L \vdash A \to C$ and $L \vdash C \to B$.

We are ready to prove LIP for $\GRC$. 
Moreover, we prove a slightly stronger result that yields LIP for $\GL$. 
Thus, the following theorem is a strengthening of Shamkanov's theorem on LIP for $\GL$. 

\begin{thm}[LIP for $\GRC$]\label{LIP_GRC}
$\GRC$ enjoys Lyndon interpolation property.
Moreover, for any $\LT$-formulas $A$ and $B$, if $\GRC \vdash A \to B$, then there exists a Lyndon interpolant $C$ of $A \to B$ in $\GRC$ satisfying the following condition:
\begin{itemize}
\item[$(\ddag)$] For $* \in \{ +, -\}$ and $D \in \MFT$, if $\bb D \in S^*(C)$, then $ \bb D \in S^*(A) \cap S^*(B)$.
\end{itemize}
\end{thm}

\begin{proof}
Towards a contradiction, suppose that there exist some $\LT$-formulas $A$ and $B$ such that $\GRC \vdash A \to B$ and $A \to B$ has no Lyndon interpolant in $\GRC$ with the condition $(\ddag)$.
For each $C \in \{ A,  \neg B\}$, let $\Phi (C)$ be the union of the following sets:
\begin{itemize}
\item
$S(C) \cup \{{\sim} B \mid B \in S(C)\}$, 
\item 
$\{ \bot, \neg \bot, \Box \bot, \neg \Box \bot \}$, 
\item 
$ \{ \Box D, \neg \Box D,
\Box \neg \bb D,
\Diamond \bb D, \Box \bb D, \neg \Box \bb D \mid \bb D \in S(C) \}$.
\end{itemize}
For pairs of finite sets of $\LT$-formulas $(\Gamma_1, \Delta_1)$ and $(\Gamma_2, \Delta_2)$, we say that $(\Gamma_2, \Delta_2)$ is an extension of $(\Gamma_1, \Delta_1)$ if $\Gamma_1 \subseteq \Gamma_2$ and $\Delta_1 \subseteq \Delta_2$. In this case, we write $(\Gamma_1, \Delta_1) \subseteq (\Gamma_2, \Delta_2)$.
We say that a pair of finite sets $(\Gamma, \Delta)$ such that $(\Gamma, \Delta) \subseteq (\Phi(A), \Phi(\neg B))$ is \textit{inseparable} if there is no $\LT$-formula $C$ satisfying $(\ddag)$ and the following conditions:
\begin{itemize}
\item $\GRC \vdash \bigwedge \Gamma \to C$, 
\item $\GRC \vdash \bigwedge \Delta \to \neg C$, 
\item $\tau(C) \subseteq \tau(A) \cap \tau( B)$.
\end{itemize}

It is easy to see that if $(\Gamma, \Delta)$ is inseparable, then $\bot \notin \Gamma \cup \Delta$. 
We say that $(\Gamma, \Delta)$ is $(A, \neg B)$-\textit{maximally inseparable} if it is maximal among inseparable pairs $(\Gamma', \Delta') \subseteq (\Phi(A), \Phi(\neg B))$. 
It is easy to show that any inseparable pair can be extended to an $(A,  \neg B)$-maximally inseparable pair.

First, we define the set $W'$ and the binary relation $\lessdot$ on $W'$ as follows: 
\begin{itemize}
\item $W' := \{ (\Gamma, \Delta) \mid$ $(\Gamma, \Delta)$ is $(A, \neg B)$-maximally inseparable$\}$, 
\item $(\Gamma_1, \Delta_1) \lessdot (\Gamma_2, \Delta_2)$ if and only if $(\Gamma_1, \Delta_1)$ and $(\Gamma_2, \Delta_2)$ satisfy the following three conditions:
\begin{enumerate}
\item $\exists C \, \bigl((\Box C \in \Gamma_2$, and $\Box C \notin \Gamma_1)$ or $(\Box C \in \Delta_2$ and $\Box C \notin \Delta_1)\bigr)$.
\item $\forall \Box C \in \Phi(A) \, (\Box C \in \Gamma_1 \Longrightarrow C, \Box C \in \Gamma_2)$.
\item $\forall \Box C \in \Phi(\neg B) \, (\Box C \in \Delta_1 \Longrightarrow C, \Box C \in \Delta_2)$.
\end{enumerate}
\end{itemize}

Note that $W'$ is a finite set. 
It is easily shown that $\lessdot$ is transitive and irreflexive. 

For each $\LT$-formula $C$, let 
\[
    \mu(C):= \{ D \mid  D \in S^+ (C)\} \cup \{ \neg D \mid D \in S^- (C)\}.
\]
The set $\mu(C)$ is almost same as $S^+(C)$, but for $\bb D \in S^-(C)$, even if $\neg \bb D \notin S^+(C)$, we have $\neg \bb D \in \mu(C)$. 
Next, we define the frame $(W, \acc, \{\prec_{D}\}_{D \in \MFT})$ as follows:

\begin{itemize}
\item $W := \{ \seq{a_0, \dots ,a_l} \mid l \geq 0$, $a_0, \ldots ,a_l \in W'$ and $\forall i < l(a_i \lessdot a_{i+1})\}$.

\item $\seq{a_0, \ldots, a_l} \acc \seq{b_0, \ldots, b_m} :\iff$ $\seq{b_0, \ldots, b_m}$ is a proper end-extension of $\seq{a_0, \ldots, a_l}$.
\item For each $D \in \MFT$, $\seq{a_0, \ldots, a_l} \prec_{D} \seq{b_0, \ldots ,b_m}$ if and only if one of the following conditions holds:
\begin{enumerate}
\item $\Box \bot \notin (a_l)_0 \cup (a_l)_1$ and $\seq{a_0, \ldots, a_l} \acc \seq{b_0, \ldots ,b_m}$.
\item $\Box \bot \in (a_l)_0 \cup (a_l)_1$, $l \leq m$, $\forall i < l (a_i = b_i)$, and 
\begin{align*}
    & \bigl[\bb D \notin \bigl((a_l)_0 \cap \mu(A) \bigr) \cup \bigl((a_l)_1 \cap \mu(\neg B)\bigr)\\ 
    & \quad \text{or}\ D \in \bigl( (b_m)_0 \cap \mu(A) \bigr) \cup \bigl((b_m)_1 \cap \mu(\neg B)\bigr) \bigr].
\end{align*}
\end{enumerate}
\end{itemize}
The set $W$ consists of all finite $\lessdot$-chains of elements of $W'$. 
Since $W'$ is finite, so is $W$.
If $l>0$, then `$l \leq m$ and $\forall i < l (a_i = b_i)$' mean that $\seq{b_0, \ldots, b_m}$ is a proper end-extension of $\seq{a_0, \ldots, a_{l-1}}$.
In this proof, we use $s_i$ for a finite sequence of $W'$. Also, elements of $W'$ are denoted by $a_i$, $b_i$ and $c_i$. Notice that $a_i \in W'$ is a pair $(\Gamma_i, \Delta_i)$ of finite sets, and so $(a_i)_0$ and $(a_i)_1$ denote $\Gamma_i$ and $\Delta_i$, respectively.

\begin{lem} \label{lem1}
The frame $(W, \acc, \{ \prec_D \}_{D \in \MFT})$ is a finite $\GRC$-frame. 
\end{lem}
\begin{proof}
It is easy to see that $\acc$ is transitive and irreflexive by the definition. 
Let $D$ be any $\LT$-formula.

\medskip
(i) $s_0 \acc s_1 \Longrightarrow s_0 \prec_D s_1$. \\
Let $\seq{a_0, \ldots ,a_l} \acc \seq{b_0, \ldots, b_m}$. Since $a_l \lessdot b_m$ and $\bot \notin (b_m)_0 \cup (b_m)_1$, we obtain $\Box \bot \notin (a_l)_0 \cup (a_l)_1$. 
By the definition of $\prec_D$, we get $\seq{a_0, \ldots, a_l} \prec_D \seq{b_0, \ldots ,b_m}$.

\medskip
(ii) $s_0 \acc s_1$ and $s_1 \prec_D s_2 \Longrightarrow s_0 \acc s_2$.\\
Assume $\seq{a_0, \ldots, a_l} \acc \seq{b_0, \ldots, b_m}$ and $\seq{b_0, \ldots, b_m} \prec_D \seq{c_0, \ldots, c_n}$.
If $\Box \bot \notin (b_m)_0 \cup (b_m)_1$, then we have $\seq{b_0, \ldots, b_m} \acc \seq{c_0, \ldots, c_n}$. Thus, we obtain $\seq{a_0, \ldots, a_l} \acc \seq{c_0, \ldots, c_n}$. If $\Box \bot \in (b_m)_0 \cup (b_m)_1$, then we have $n \geq m$ and $\forall i < m (b_i = c_i)$. Since $l < m$ and $\forall i \leq l (a_i = b_i)$, we obtain $l<n$ and $\forall i \leq l (a_i = c_i)$. Hence, $\seq{a_0, \ldots, a_l} \acc \seq{c_0, \ldots, c_n}$.

\medskip
(iii) $s_0 \prec_D s_1$ and $s_0 \acc s_2 \Longrightarrow s_0 \acc s_1$. \\
Suppose $\seq{a_0, \ldots, a_l} \prec_D \seq{b_0, \ldots, b_m}$ and $\seq{a_0, \ldots, a_l} \acc \seq{c_0, \ldots, c_n}$. 
Since $a_l \lessdot c_n$ and $\bot \notin (c_n)_0 \cup (c_n)_1$, we obtain $\Box \bot \notin (a_l)_0 \cup (a_l)_1$. 
We get $\seq{a_0, \ldots, a_l} \acc \seq{b_0, \ldots, b_m}$.

\medskip
(iv) $s_0 \acc s_1 \Longrightarrow \exists s_2(s_0 \acc s_2$ and $s_1 \prec_D s_2)$. \\
If $\seq{a_0, \ldots, a_l} \acc \seq{b_0, \ldots, b_m}$, then we obtain $l < m$ and $\forall i \leq l(a_i = b_i)$. We distinguish the following two cases:

\paragraph{Case 1:} $\Box \bot \notin (b_m)_0 \cup (b_m)_1$. \\
Then, $\neg \Box \bot \in (b_m)_0 \cup (b_m)_1$.
Since $b_m$ is $(A, \neg B)$-maximally inseparable, we have $\neg \Box \bot \in (b_m)_0$. 
Let 
\begin{itemize}
    \item $\Gamma_1 = \{ \Box \bot \} \cup \{ \Box C, C \mid \Box C \in (b_m)_0\}$, 
    \item $\Delta_1 =\{ \Box D, D \mid \Box D \in (b_m)_1 \}$.
\end{itemize}
Towards a contradiction, assume that $(\Gamma_1, \Delta_1)$ is separable. 
Then, there exists an $\LT$-formula $E$ satisfying $(\ddag)$ and the following conditions: 
\begin{itemize}
    \item $\GRC \vdash \bigwedge_{\Box C \in (b_m)_0} \boxdot C \wedge \Box \bot \to E$, 
    \item $\GRC  \vdash \bigwedge_{\Box D \in (b_m)_1} \boxdot D \to \neg E$, 
    \item $\tau(E) \subseteq \tau(A) \cap \tau(B)$, 
\end{itemize}
where $\boxdot C$ denotes $\Box C \wedge C$.
By propositional logic, 
\[
    \displaystyle \GRC \vdash \bigwedge_{\Box C \in (b_m)_0} \boxdot C \to (\neg E \to (\Box \bot \to \bot)).
\]
Then,  
\[
    \displaystyle \GRC  \vdash \bigwedge_{\Box C \in(b_m)_0} \Box \boxdot C \to \bigl(\Box \neg E \to \Box(\Box \bot \to \bot) \bigr).
\]
Since $\GL \vdash \Box(\Box \bot \to \bot) \to \Box \bot$ and $\GL \vdash \Box C \to \Box \Box C$,
we obtain 
\[
    \displaystyle \GRC  \vdash \bigwedge_{\Box C \in (b_m)_0} \Box C \wedge \neg \Box \bot \to \Diamond E.
\]
On the other hand, $\displaystyle \GRC \vdash \bigwedge_{\Box D \in (b_m)_1} \Box D \to \Box \neg E$.
Since $\tau(\Diamond E)=\tau(E) \subseteq \tau(A) \cap \tau(B)$ and $\Diamond E$ satisfies $(\ddag)$, this contradicts the inseparability of $b_m$.

Hence, $(\Gamma_1, \Delta_1)$ is inseparable and it can be extended to some $d \in W'$. 
By the definition of $(\Gamma_1, \Delta_1)$, $b_m \lessdot d$ holds. 
Hence, we obtain $\seq{b_0, \ldots, b_m} \acc \seq{b_0, \ldots, b_m, d}$. 
Therefore,  $\seq{a_0, \ldots, a_l} \acc \seq{b_0, \ldots, b_m, d}$ and
$\seq{b_0, \ldots, b_m} \prec_{D} \seq{b_0, \ldots, b_m, d}$ hold.

\paragraph{Case 2:} $\Box \bot \in (b_m)_0 \cup (b_m)_1$. \\
Since $l<m$, we have $m \geq 1$.
Suppose $\bb D \in  \bigl((b_m)_0 \cap \mu(A) \bigr) \cup \bigl((b_m)_1 \cap \mu(\neg B) \bigr)$. 
We only consider the case that $\bb D \in  (b_m)_0 \cap \mu(A)$.
The case that $\bb D \in 
(b_m)_1 \cap \mu(\neg B)$
is proved similarly.
Let us find some $d \in W'$ such that $D \in (d)_0$ and $\seq{b_0, \dots, b_m} \prec_D \seq{b_0, \ldots, b_{m-1}, d}$.

Since $\bb D \in (b_m)_0 \subseteq \Phi(A)$, we have $\Box \neg D, \Box \neg \bb D \in \Phi(A)$.
Here, we prove $\Diamond D \in (b_{m-1})_0$.
Assume otherwise, then we have $\Box \neg D \in (b_{m-1})_0$. 
Hence, by $\GRC \vdash \Box \neg D \to \Box \neg \bb D$, it follows that $\Box \neg \bb D \in (b_{m-1})_0$. Because $b_{m-1} \lessdot b_m$, we have $\neg \bb D \in (b_m)_0$.
This contradicts $\bot \notin (b_m)_0$. 
We obtain $\Diamond D \in (b_{m-1})_0$.

Next, let 
\begin{itemize}
    \item $\Gamma_2 = \{ \Box \neg D, D\} \cup \{ \Box C, C \mid \Box C \in (b_{m-1})_0 \}$, 
    \item $\Delta_2 = \{ \Box E, E \mid \Box E \in (b_{m-1})_1\}$. 
\end{itemize}
As in Case 1, we can prove that $(\Gamma_2, \Delta_2)$ is inseparable.
We find some $d \in W'$ which is an extension of $(\Gamma_2, \Delta_2)$.
By the definition of $(\Gamma_2, \Delta_2)$, we obtain $b_{m-1} \lessdot d$ and $\seq{b_0, \ldots, b_{m-1}, d} \in W$.
Since $\Box \bot \in (b_m)_0 \cup (b_m)_1$ and $D \in (d)_0 \cap \mu(A)$, we have $\seq{b_0, \ldots, b_m} \prec_D \seq{b_0, \ldots, b_{m-1}, d}$. 
Also, since $l \leq m-1$, we obtain $\seq{a_0, \ldots ,a_l} \acc \seq{b_0, \ldots, b_{m-1}, d}$. 
\end{proof}

For each propositional variable $p$ and $\seq{a_0, \ldots, a_l} \in W$, we define
\begin{equation*}
\seq{a_0, \ldots,a_l} \Vdash p :\iff p \in \bigl((a_l)_0 \cap \mu(A)\bigr) \cup \bigl((a_l)_1 \cap \mu(\neg B)\bigr).
\end{equation*}

\begin{lem} \label{lem2}
For any $\LT$-formula $C$ and
$\seq{a_0, \ldots, a_l} \in W$, \\
if $C \in \bigl((a_l)_0 \cap \mu(A)\bigr) \cup \bigl((a_l)_1 \cap \mu(\neg B)\bigr)$, then $\seq{a_0, \ldots, a_l} \Vdash C$.
\end{lem}
\begin{proof}
Let $C$ be an $\LT$-formula and $\seq{a_0, \ldots, a_l} \in W$. 
We prove the lemma by induction on the number of occurrences of the operators $\neg, \vee, \Box$, and $\bb$ in $D$.
Suppose $C \in ((a_l)_0 \cap \mu(A)) \cup ((a_l)_1 \cap \mu(\neg B))$.
We only consider the case that $C \in (a_l)_0 \cap \mu(A)$. 
The case that $C \in (a_l)_1 \cup \mu(\neg B)$ is proved in the similar way. 

\medskip
$C \equiv p$ or $\bot$: This case is obvious by the definition and $\bot \notin (a_l)_0$. 

\medskip
$C \equiv D_1 \vee D_2$ or $\Box D$: Obviously the claim holds by the induction hypothesis.

\medskip
$C \equiv \bb D$: 
Suppose $\bb D \in (a_l)_0 \cap \mu(A)$.
Take $\seq{b_0, \ldots, b_m} \in W$ with $\seq{a_0, \ldots, a_l} \prec_D \seq{b_0, \ldots, b_m}$ arbitrarily.
Since $\bb D \in (a_l)_0 \subseteq \Phi(A)$, we have $\Box D \in \Phi(A)$ and $\Box D \in (a_l)_0$ by $\GRC \vdash \bb D \to \Box D$.
If $\Box \bot \notin (a_l)_0 \cup (a_l)_1$, $\seq{a_0, \ldots, a_l} \acc \seq{b_0, \ldots, b_m}$ holds. Since $a_l \lessdot b_m$, we have $D \in (b_m)_0$. Since $D \in \mu (A)$, by the induction hypothesis, we obtain $\seq{b_0, \ldots,b_m} \Vdash D$.
If $\Box \bot \in (a_l)_0 \cup (a_l)_1$, since $\seq{a_0, \ldots, a_l} \prec_D \seq{b_0, \ldots, b_m}$, we have $D \in (b_m)_0 \cap \mu (A)$ or $D \in (b_m)_1 \cap \mu (\neg B)$. By the induction hypothesis, in both cases, it follows $\seq{b_0, \ldots, b_m} \Vdash D$.
Consequently, we obtain $\seq{a_0, \ldots, a_l} \Vdash \bb D$. 

\medskip
$C \equiv \neg D'$: We distinguish the following four cases.

\medskip
(i) $D' \equiv  p$ : 
Suppose $\neg p \in (a_l)_0  \cap \mu(A)$.
Towards a contradiction, assume $\seq{a_0, \ldots, a_l} \Vdash p$. 
Then, we obtain $p \in (a_l)_0  \cap \mu(A)$ or $p \in (a_l)_1 \cap \mu(\neg B)$. 
In the former case, since $\neg p \in (a_l)_0$, we obtain $\GRC \vdash  \bigwedge (a_l)_0 \to \bot$. 
This contradicts the inseparability of $a_l$. 

In the latter case, it follows $\GRC \vdash \bigwedge (a_l)_1 \to p$. 
By the supposition, $\GRC \vdash \bigwedge (a_l)_0 \to \neg p$. 
Here, $\tau(\neg p) = \{\neg p\}$. 
Since $\neg p \in \mu(A)$, we have $p \in S^-(A)$, and hence $\neg p \in \tau(A)$. 
Also, $p \in \mu(\neg B)$ implies $p \in S^+(\neg B)$, and thus we get $\neg p \in \tau(B)$. 
Therefore, $\tau(\neg p) \subseteq \tau(A) \cap \tau(B)$. 
Since $\tau(\neg p)$ is $\bb$-free, $\neg p$ satisfies $(\ddag)$. 
Therefore, this contradicts the inseparability of $a_l$. 
Hence, we conclude $\seq{a_0, \ldots, a_l} \Vdash \neg p$.

\medskip
(ii) $D' \equiv  D_1 \vee D_2$ : We can easily obtain the lemma by the induction hypothesis.

\medskip
(iii) $D' \equiv  \Box D$ : Suppose $\neg \Box D \in (a_l)_0 \cap \mu(A) $.
Let 
\begin{itemize}
    \item $\Gamma_1 := \{ \neg D, \Box D\} \cup \{ \Box E, E \mid \Box E \in (a_l)_0 \}$, 
    \item $\Delta_1 := \{ \Box F, F \mid \Box F \in (a_l)_1\}$. 
\end{itemize}
In the similar argument as in Case 1 of the proof of Lemma \ref{lem1}, we can prove that $(\Gamma_1, \Delta_1)$ is inseparable.
Let $d \in W'$ be an extension of $(\Gamma_1, \Delta_1)$. 
We obtain $a_l \lessdot d$ and $\seq{a_0, \ldots, a_l, d} \in W$.
Since $\neg D \in  (d)_0 \cap \mu(A)$, by the induction hypothesis, we obtain $\seq{a_0, \ldots, a_l, d} \Vdash \neg D$. Since 
$\seq{a_0, \ldots, a_l} \acc \seq{a_0, \ldots, a_l, d}$,
we conclude that $\seq{a_0, \ldots, a_l} \Vdash \neg \Box D$. 

\medskip
(iv) $D' \equiv  \bb D$ : 
Suppose $\neg \bb D \in   (a_l)_0 \cap \mu(A) $. We distinguish the following two cases.

\paragraph{Case 1:} $\Box \bot \notin (a_l)_0 \cup (a_l)_1$. \\
Then, $\neg \Box \bot \in (a_l)_0 \cup (a_l)_1$. 
Since $\neg \bb D \in (a_l)_0$, we have $\Box D \in \Phi(A)$. 
It follows from $\GRC \vdash \Box D \to (\Box \bot \vee \bb D)$ that $\neg \Box D \in (a_l)_0$. 
Let 
\begin{itemize}
    \item $\Gamma_2 := \{ \neg D, \Box D\} \cup \{ \Box E, E \mid \Box E \in (a_l)_0 \}$, 
    \item $\Delta_2 := \{ \Box F, F \mid \Box F \in (a_l)_1\}$. 
\end{itemize}
In the similar argument as in the case that $D' \equiv \Box D$, we can prove that $(\Gamma_2, \Delta_2)$ is inseparable.
Let $d \in W'$ be an extension of $(\Gamma_2, \Delta_2)$. 
We get $a_l \lessdot d$ and $\seq{a_0, \ldots, a_l, d} \in W$.
Since $\neg D \in  (d)_0 \cap \mu(A)$, by the induction hypothesis, we obtain $\seq{a_0, \ldots, a_l, d} \Vdash \neg D$. 
Because $\Box \bot \notin (a_l)_0 \cup (a_l)_1$ and $\seq{a_0, \ldots, a_l} \acc \seq{a_0, \ldots, a_l, d}$, we have $\seq{a_0, \ldots, a_l} \prec_D \seq{a_0, \ldots, a_l, d}$. Hence, 
we conclude that $\seq{a_0, \ldots, a_l} \Vdash \neg \bb D$.

\paragraph{Case 2:} $\Box \bot \in (a_l)_0 \cup (a_l)_1$. \\
If $\bb D \in (a_l)_1 \cap \mu(\neg B)$, then $\bb D \in S^-(A) \cap S^-(B)$ because $\neg \bb D \in   \mu(A)$, and so we would have $\tau(\neg \bb D) \subseteq \tau(A) \cap \tau(B)$ and $\neg \bb D$ would satisfy $(\ddag)$. 
This contradicts the inseparability of $a_l$. 
Thus, we have 
\begin{align}\label{notin}
    \bb D \notin \bigl((a_l)_0 \cap \mu(A) \bigr) \cup \bigl((a_l)_1 \cap \mu(\neg B) \bigr).
\end{align}
We distinguish the following two cases.

\paragraph{Case 2-1:} $l \geq 1$. \\
Since $\bb D \notin (a_{l})_0$, we have $\Box \bb D \notin (a_{l-1})_0$.
 By $\GRC \vdash \Box D \to \Box \bb D$, we obtain $\Box D \notin (a_{l-1})_0$. 
Let
\begin{itemize}
    \item $\Gamma_3 := \{ \Box D, \neg D \} \cup \{ \Box E,  E \mid \Box E \in (a_{l-1})_0 \}$,
    \item $\Delta_3 := \{ \Box F, F \mid \Box F \in (a_{l-1})_1 \}$. 
\end{itemize}
By the similar argument as in the proof of Lemma \ref{lem1}, we obtain that $(\Gamma_3, \Delta_3)$ is inseparable.
Let $d \in W$ be an extension of $(\Gamma_3, \Delta_3)$.  We obtain $a_{l-1} \lessdot d$ and $\seq{a_0, \ldots, a_{l-1}, d} \in W$.
Since $\neg D \in   (d)_0 \cap \mu(A)$, by the induction hypothesis, we obtain $\seq{a_0, \ldots, a_{l-1}, d} \Vdash \neg D$. 
By (\ref{notin}), we have $\seq{a_0, \ldots, a_l} \prec_D \seq{a_0, \ldots, a_{l-1}, d}$ by the definition of $\prec_D$. 
We conclude that $\seq{a_0, \ldots, a_l} \Vdash \neg \bb D$. 

\paragraph{Case 2-2:} $l=0$. \\
Suppose that $(\{\neg D\}, \emptyset)$ is separable. 
Then, we would have $\GRC \vdash D$ and $\GRC \vdash \bb D$. 
This contradicts the inseparability of $a_l$ because $\neg \bb D \in (a_l)_0$.  
Hence, $(\{ \neg D\}, \emptyset)$ is inseparable. 
Let $d \in W'$ be an extension of $(\{\neg D\} , \emptyset)$. 
Then, we obtain $\seq{a_l} \prec_D \seq{d}$ by (\ref{notin}). 
Since $\neg D \in (d)_0 \cap \mu(A)$, by the induction hypothesis, we obtain $\seq{d} \Vdash \neg D$. 
Therefore, $\seq{a_l} \Vdash \neg \bb D$.
\end{proof}

We resume the proof of Theorem \ref{LIP_GRC}. Since $A \to B$ has no Lyndon interpolant in $\GRC$ with $(\ddag)$, the pair $(\{A\}, \{ \neg B\})$ is inseparable. 
Let $d \in W'$ be an extension of $(\{A\}, \{ \neg B\})$, then $\seq{d} \in W$. 
Since $A \in (d)_0 \cap \mu(A)$ and $\neg B \in (d)_1 \cap \mu(\neg B)$, by Lemma \ref{lem2}, we obtain $\seq{d} \Vdash A \wedge \neg B$. 
By Proposition \ref{soundness_GRC}, we have $\GRC \nvdash A \to B$, a contradiction. 
\end{proof}

\begin{thm}[The finite frame property of $\GRC$]\label{ffp_GRC}
For any $\LT$-formula $A$, the following are equivalent: 
\begin{enumerate}
    \item $\GRC \vdash A$. 
    \item $A$ is valid on all $\GRC$-frames. 
    \item $A$ is valid on all finite $\GRC$-frames. 
\end{enumerate}
\end{thm}
\begin{proof}
$(1 \Rightarrow 2)$ and $(2 \Rightarrow 3)$ are obvious. 
We prove the contrapositive of $(3 \Rightarrow 1)$. 
Suppose $\GRC \nvdash A$. 
 Then, $(\{ \neg A\}, \emptyset)$ is inseparable.
Thus, by our proof of Theorem \ref{LIP_GRC}, there exists some finite $\GRC$-model in which $A$ is not valid.
\end{proof}

\begin{cor}[LIP for $\GL$ (Shamkanov {\cite[Theorem 2.10]{Shamkanov}})]\label{LIP_GL}
The logic $\GL$ enjoys Lyndon interpolation property. 
\end{cor}
\begin{proof}
Suppose $\GL \vdash A \to B$ for $A, B \in \LBox$. 
We obtain $\GRC \vdash A \to B$.
By Theorem \ref{LIP_GRC}, $A \to B$ has a Lyndon interpolant $C$ with the condition $(\ddag)$. 
That is, $\GRC \vdash A \to C$, $\GRC \vdash C \to B$, $\tau(C) \subseteq \tau(A) \cap \tau(B)$, and every $\bb D \in S(C)$ is contained in $S(A) \cup S(B)$. 
Since $A$ and $B$ are $\bb$-free, so is $C$. 
Thus, by Fact \ref{conservation_GR}, we obtain $\GL \vdash A \to C$, $\GL \vdash C \to B$ and $\tau(C) \subseteq \tau(A) \cap \tau(B)$.
\end{proof}


\section{Lyndon interpolation property for $\GR$}\label{sec_LIP_GR}

The purpose of this section is to prove Lyndon interpolation property for $\GR$. 
Besides, we introduce the logic $\GRB$ and show that $\GRB$ provides an alternative axiomatization of $\GR$.
To do so, we first prove LIP of $\GRB$ and then prove the equivalence of $\GRB$ and $\GR$.
Moreover, we obtain the finite frame property of $\GR$ with respect to non-trivial $\GR$-frames. 

\begin{defn}
    The logic $\GRB$ is obtained from $\GRC$ by adding the Rosser rule $\dfrac{\neg A}{\neg \bb A}$. 
\end{defn}

Before proving LIP for $\GRB$, we prepare some observation on $\GRC$-frames.  

\begin{defn}\leavevmode
\begin{itemize}
    \item A $\GRC$-frame $(W, \acc, \{\prec_B\}_{B \in \MFT})$ is said to be \textit{non-trivial} iff there exist $r, w \in W$ such that $r$ is the root with respect to $\acc$ and $r \neq w$. 
    
    \item We say that a $\GRC$-frame $(W, \acc, \{\prec_B\}_{B \in \MFT})$ is \textit{$\bb$-serial} iff for any $\LT$-formula $B$ and $x \in W$, there exists a $y \in W$ such that $x \prec_B y$. 
\end{itemize}
\end{defn}

\begin{prop}\label{bb-serial}\leavevmode
\begin{enumerate}
    \item For any $\LT$-formula $A$, if $\GRB \vdash A$, then $A$ is valid on all $\bb$-serial $\GRC$-frames. 
    \item Every non-trivial $\GRC$-frame is $\bb$-serial. 
\end{enumerate}
\end{prop}
\begin{proof}
1. This is proved by induction on the length of proofs of $A$ in $\GRB$. 
We leave the proof to the reader. 

2. Let $(W, \acc, \{\prec_B\}_{B \in \MFT})$ be any non-trivial $\GRC$-frame with the root $r \in W$. 
    We show that $\prec_B$ is serial for all $B \in \MFT$. 
    Let $x$ be any element of $W$. 
    We distinguish the following two cases: 

    (i) $x = r$: Since the frame is non-trivial, we find an element $w \in W$ such that $r \acc w$. 
    Then, we have $r \prec_B w$ for any $B \in \MFT$.

    (ii) $x \neq r$: Since $r \acc x$ and the frame is a $\GRC$-frame, for any $B \in \MFT$, there exists a $y \in W$ such that $x \prec_B y$ and $r \acc y$. 
\end{proof}

We are ready to prove LIP for $\GRB$. 

\begin{thm}[LIP for $\GRB$]\label{LIP_GRB}
$\GRB$ enjoys Lyndon interpolation property. 
Moreover, for any $\LT$-formulas $A$ and $B$, if $\GRB \vdash A \to B$, then there exists a Lyndon interpolant $C$ of $A \to B$ in $\GRB$ satisfying the condition $(\ddag)$ in Theorem \ref{LIP_GRC}. 
\end{thm}
\begin{proof}
The proof of the theorem is essentially same as that of Theorem \ref{LIP_GRC}. 
So, we only give an outline of the proof. 

Towards a contradiction, suppose that there exist some $\LT$-formulas $A$ and $B$ such that $\GR \vdash A \to B$ and $A \to B$ has no Lyndon interpolant in $\GR$ with the condition $(\ddag)$.
We can define the notion of inseparability of pairs of finite sets of $\LT$-formulas as in the proof of Theorem \ref{LIP_GRC} based on $\GRB$ instead of $\GRC$. 
Let $W'$ be the set of all $(A, \neg B)$-maximally inseparable pairs. 
We can define the transitive and irreflexive binary relation $\lessdot$ on $W'$ as in the proof of Theorem \ref{LIP_GRC}. 
Then, we define the desired model $(W, \acc, \{\prec_D\}_{D \in \MFT}, \Vdash)$ also as in the proof of Theorem \ref{LIP_GRC}, except that $W$ in this proof contains the empty sequence $\epsilon$ as the root element. 
We can prove that for any $\LT$-formula $C$ and $\seq{a_0, \ldots, a_l} \in W \setminus \{\epsilon \}$, if $C \in \bigl((a_l)_0 \cap \mu(A)\bigr) \cup \bigl((a_l)_1 \cap \mu(\neg B)\bigr)$, then $\seq{a_0, \ldots, a_l} \Vdash C$ in the same way as in the proof of Lemma \ref{lem2}.
So, it suffices to prove that our frame is actually a $\GRC$-frame. 
This is proved in the almost same way as in the proof of Lemma \ref{lem1}, except that $\epsilon$ is considered. 
We prove only the following essential condition in the consideration of $\epsilon$: 
\begin{align}\label{cond}
    \forall D \in \MFT\, \forall y \in W\, \bigl(\epsilon \acc y \Rightarrow \exists z \in W\, (y \prec_D z\ \&\ \epsilon \acc z) \bigr). 
\end{align}
Suppose $\epsilon \acc \seq{a_0, \ldots, a_l}$. 
If $\Box \bot \notin (a_l)_0 \cup (a_l)_1$, then it is proved that there exists a $d \in W'$ such that $\seq{a_0, \ldots, a_l} \prec_D \seq{a_0, \ldots, a_{l}, d}$ (cf.~Case 1 in the proof of Lemma \ref{lem1}).
If $l \geq 1$ and $\Box \bot \in (a_l)_0 \cup (a_l)_1$, then it is shown that there exists a $d \in W'$ such that $\seq{a_0, \ldots, a_l} \prec_D \seq{a_0, \ldots, a_{l-1}, d}$ (cf.~Case 2 in the proof of Lemma \ref{lem1}).

So, we may assume that $l = 0$ and $\Box \bot \in (a_l)_0 \cup (a_l)_1$. 
Suppose $\bb D \in \Bigl( \bigl( (a_l)_0 \cap \mu(A) \bigr) \cup \bigl((a_l)_1 \cap \mu(\neg B) \bigr) \Bigr)$. 
We only give a proof of the case $\bb D \in (a_l)_0 \cap \mu(A)$, and the case $\bb D \in (a_l)_1 \cap \mu(\neg B)$ is proved similarly. 
Towards a contradiction, we assume that $(\{D\}, \emptyset)$ is separable. 
Then, $\GRB \vdash \neg D$, and so $\GRB \vdash \neg \bb D$. 
This contradicts the inseparability of $a_l$. 
Therefore, $(\{D\}, \emptyset)$ is inseparable. 
We find some $d \in W'$ which extends $(\{D\}, \emptyset)$. 
Since $D \in (d)_0 \cap \mu(A)$, we obtain $\seq{a_l} \prec_D \seq{d}$. 
This completes the proof of the condition (\ref{cond}). 

Since $A \to B$ has no Lyndon interpolant with $(\ddag)$ in $\GRB$, the pair $(\{A\}, \{ \neg B\})$ is inseparable. 
Let $d \in W'$ be an extension of $(\{A\}, \{ \neg B\})$. 
Then $\seq{d} \in W$ and so our model $(W, \acc, \{\prec_D\}_{D \in \MFT}, \Vdash)$ is a non-trivial $\GRC$-model. 
Also, we have $\seq{d} \Vdash A \wedge \neg B$. 
By Proposition \ref{bb-serial}.(2), our model is $\bb$-serial. 
So by Proposition \ref{bb-serial}.(1), every theorem of $\GRB$ is valid in the model. 
Thus, $\GRB \nvdash A \to B$, and this is a contradiction. 
\end{proof}


As a consequence of our proof of Theorem \ref{LIP_GRB}, we obtain the finite frame property of $\GR$ with respect to $\GRC$-frames. 
As a byproduct, we also obtain that $\GRB$ gives an alternative axiomatization of $\GR$. 

\begin{thm}[The finite frame property of $\GR$]\label{ffp_GR}
For any $\LT$-formula $A$, the following are equivalent: 
\begin{enumerate}
    \item $\GRB \vdash A$. 
    \item $\GR \vdash A$. 
    \item $A$ is valid on all non-trivial $\GRC$-frames. 
    \item $A$ is valid on all $\bb$-serial $\GRC$-frames. 
\end{enumerate}
\end{thm}
\begin{proof}
    $(1 \Rightarrow 2)$: This is because the rules $\dfrac{B}{\bb B}$ and $\dfrac{\neg B}{\neg \bb B}$ are admissible in $\GR$. 

    $(2 \Rightarrow 3)$: By Proposition \ref{soundness_GRC}, every $\GRC$-model can be considered as a $\GRM$-model. 
So, every non-trivial $\GRC$-model can be considered as a non-trivial $\GRM$-model. 
Hence, this implication follows from Fact \ref{compl_GRM}.(2). 
    
    $(3 \Rightarrow 4)$: Suppose that $A$ is valid on all non-trivial $\GRC$-frames. 
Let $(W, \acc, \{\prec_B\}_{B \in \MFT})$ be any $\bb$-serial $\GRC$-frame.  
We prepare a new element $r$ and define the following new frame $(W', \acc', \{\prec'_B\}_{B \in \MFT})$ as follows:
\begin{itemize}
    \item $W' = W \cup \{r\}$, 
    \item $\acc' = \acc \cup \{(r, x) \mid x \in W\}$, 
    \item $\prec_B' = \prec_B \cup \{(r, x) \mid x \in W\}$ for $B \in \MFT$. 
\end{itemize}
It is easy to show that $(W', \acc', \{\prec'_B\}_{B \in \MFT})$ is a $\GRC$-frame. 
We only give a proof of the property: 
\[
    \forall y \in W'\, (r \acc' y \Rightarrow \exists z \in W'\, (y \prec_B' z\ \&\ r \acc' z)).
\]
Suppose $r \acc' y$.  
By the $\bb$-seriality, we find some $z \in W$ such that $y \prec_B z$. 
Then, we have $y \prec_B' z$ and $r \acc' z$. 

Since $(W', \acc', \{\prec'_B\}_{B \in \MFT})$ is non-trivial, $A$ is valid on the frame. 
So, $A$ is also valid on $(W, \acc, \{\prec_B\}_{B \in \MFT})$. 

    $(4 \Rightarrow 3)$: This implication follows from Proposition \ref{bb-serial}.(2).     
    
    $(3 \Rightarrow 1)$: Suppose $\GRB \nvdash A$. 
 Then, as in the proof of Theorem \ref{ffp_GRC}, from our proof of Theorem \ref{LIP_GRB}, we obtain a finite non-trivial $\GRC$-model in which $A$ is not valid.
\end{proof}

From Theorems \ref{LIP_GRB} and \ref{ffp_GR}, we obtain LIP for $\GR$. 

\begin{cor}[LIP for $\GR$]
$\GR$ enjoys Lyndon interpolation property. 
\end{cor}

\section{Uniform interpolation property}\label{sec_UIP}

In this section, we prove that the logics $\GRM$, $\GRC$, and $\GR$ enjoy uniform interpolation property. 

\begin{defn}
We say that a logic $L$ enjoys \textit{uniform interpolation property} iff for any formula $A$ and any finite set $P$ of propositional variables, there exists a formula $C$ satisfying the following conditions:
\begin{enumerate}
    \item $v(C) \subseteq v(A) \setminus P$, 
    \item $L \vdash A \to C$, 
    \item for any formula $B$, if $L \vdash A \to B$ and $v(B) \cap P = \emptyset$, then $L \vdash C \to B$. 
\end{enumerate}
Such a formula $C$ is called a uniform interpolant of $(A, P)$ in $L$. 
\end{defn}

\begin{fact}[Shavrukov \cite{Shavrukov93}]
$\GL$ enjoys uniform interpolation property. 
\end{fact}

\subsection{UIP for $\GRM$}

Firstly, we prove that $\GRM$ enjoys UIP in the same way as in Sidon's proof of CIP for $\GRM$. 
Here, we fix an $A \in \MFT$ and a finite set $P$ of propositional variables. 
We define the set $S_0(A)$ of all $\bb$-outermost-subformulas of $A$ recursively as follows: 
\begin{itemize}
    \item $S_0(B) = \{B\}$ when $B$ is either a propositional variable, $\bot$, or of the form $\bb C$, 
    \item $S_0(B \lor C) = S_0(B) \cup S_0(C) \cup \{B \lor C\}$, 
    \item $S_0(\ast B) = S_0(B) \cup \{\ast B\}$ for $\ast \in \{\neg, \Box\}$. 
\end{itemize}
For each $\bb D \in S_0(A)$, we prepare a propositional variable $q_D$ not occurring in $A$ and $P$. 
We define the translation $\dagger$ from $S_0(A)$ to $\MF_\Box$ recursively as follows:  
\begin{itemize}
\item $p^\dagger$ is $p$ for each propositional variable $p$, 
\item $\bot^\dagger$ is $\bot$, 
\item $\dagger$ preserves $\Box$ and propositional connectives, 
\item $(\bb D)^\dagger$ is $q_D$. 
\end{itemize}
For each $\bb D \in S_0(A)$, let $\Psi_D$ be the conjunction of the four $\MF_\Box$-formulas $q_D \to \Box D^\dagger$, $q_D \to \Box q_D$, $\Box D^\dagger \to (\Box \bot \vee q_D)$, and $\Diamond q_D \to \Diamond D^\dagger$. 
The following fact is implicitly stated in Shavrukov \cite[proof of Theorem 1.8]{Shavrukov} and Sidon \cite[proof of Theorem 2.1]{Sidon}, and is easily proved by using Kripke completeness of $\GL$ and Lemma 1.6 of \cite{Shavrukov}. 

\begin{fact}\label{dagger}
For any $\LT$-formula $A$,
\[
    \GRM \vdash A\ \text{if and only if}\ \GL \vdash \bigwedge_{\bb D \in S_0(A)} \boxdot \Psi_D \to A^\dagger.
\]
\end{fact}

\begin{thm}[UIP for $\GRM$] \label{UIP_GRM}
$\GRM$ enjoys uniform interpolation property.
\end{thm}
\begin{proof} 
We fix an $A \in \MFT$ and a finite set $P$ of propositional variables. 
Let $Q : = P \cup \{q_D \mid \bb D \in S_0(A)$ and $v(D) \cap P \neq \emptyset\}$. 
By applying UIP for $\GL$, we find a uniform interpolant $C$ of $\bigl(\bigwedge_{\bb D \in S_0(A)} \boxdot \Psi_D \land A^\dagger, Q \bigr)$ in $\GL$. 
That is, 
\begin{enumerate}
\item $v(C) \subseteq v\bigl(\bigwedge_{\bb D \in S_0(A)} \boxdot \Psi_D \land A^\dagger \bigr) \setminus Q$, 
\item $\GL \vdash \bigwedge_{\bb D \in S_0(A)} \boxdot \Psi_D \land A^\dagger \to C$, 
\item for any $B \in \MF_\Box$, if $\GL \vdash \bigwedge_{\bb D \in S_0(A)} \boxdot \Psi_D \land A^\dagger \to B$ and $v(B) \cap Q = \emptyset$, then $\GL \vdash C \to B$.
\end{enumerate}
Here, $v\bigl(\bigwedge_{\bb D \in S_0(A)} \boxdot \Psi_D \land A^\dagger \bigr) \setminus Q$ is
\[
  (v(A) \setminus P) \cup \{q_D \mid \bb D \in S_0(A)\ \text{and}\ v(D) \cap P = \emptyset\}. 
\]
Let $\sigma$ be a uniform substitution so that $\sigma(q_D) = \bb D$ for each $\bb D \in S_0(A)$. 
We prove that $\sigma(C)$ is a uniform interpolant of $(A, P)$ in $\GRM$.

\begin{itemize}
\item Since we have $v(\bb D) \subseteq v(A) \setminus P$ for each $\bb D \in S_0(A)$ with $v(D) \cap P = \emptyset$, we obtain $v(\sigma(C)) \subseteq v(A) \setminus P$.

\item  Since $\GRM$ is an extension of $\GL$, we have $\GRM \vdash \bigwedge_{\bb D \in S_0(A)} \boxdot \Psi_D \land A^\dagger \to C$. 
Then,
\[
    \GRM \vdash \bigwedge_{\bb D \in S_0(A)} \sigma(\boxdot \Psi_D) \land \sigma(A^\dagger) \to \sigma(C). 
\]
It is easy to see that $\GRM \vdash \sigma(\boxdot \Psi_D)$ and $\sigma(A^\dagger) = A$. 
Hence, $\GRM \vdash A \to \sigma(C)$. 

\item Let $B$ be any $\LT$-formula such that $\GRM \vdash A \to B$ and $v(B) \cap P = \emptyset$. 
Then, by Fact \ref{dagger}, 
\[
    \GL \vdash \bigwedge_{\bb D \in S_0(A \to B)} \boxdot \Psi_D \to (A^\dagger \to B^\dagger).
\]
Here, we may assume $v\bigl(\bigwedge_{\bb D \in S_0(A \to B)} \Psi_D \bigr) \cap P = \emptyset$. 
Then, we obtain 
\[
    \GL \vdash \bigwedge_{\bb D \in S_0(A)} \boxdot \Psi_D \wedge A^\dagger \to \bigl(\bigwedge_{\bb E \in S_0(B)} \boxdot \Psi_E \to B^\dagger \bigr).
\]
Since $v(B) \cap P = \emptyset$, for each $\bb E \in S_0(B)$, we have $v(E) \cap P = \emptyset$. 
Hence, $v\bigl(\bigwedge_{\bb E \in S_0(B)} \boxdot \Psi_E \to B^\dagger \bigr) \cap Q = \emptyset$. 
Since $C$ is an appropriate uniform interpolant, $\GL \vdash C \to \bigl(\bigwedge_{\bb E \in S_0(B)} \boxdot \Psi_E \to B^\dagger \bigr)$, and this formula is also provable in $\GRM$. 
By applying $\sigma$, we have $\GRM \vdash \sigma(C) \to (\bigwedge_{\bb E \in S_0(B)} \sigma(\boxdot \Psi_E) \to \sigma(B^\dagger))$. 
We conclude $\GRM \vdash \sigma(C) \to B$. \qedhere
\end{itemize}
\end{proof}

\subsection{UIP for $\GRC$}

Secondly, we prove UIP for $\GRC$ by embedding $\GRC$ into $\GRM$. 
Let $A^\top$ be the $\LT$-formula obtained from $A$ by replacing all $\bb D \in S_0(A)$ satisfying $\GRC \vdash D$ by $\neg \bot$. 
For example, for a formula $D$ with $\GRC \vdash D$, we have that $(\Box \bb D \lor \bb p)^\top$ is $\Box \neg \bot \lor \bb p$. 
It is easy to show that $\GRC \vdash A \leftrightarrow A^\top$. 
The following proposition is inspired by Sidon's technique on the embedding of $\GR$ into $\GRM$ (cf.~\cite[Theorem 2.2]{Sidon}). 

\begin{prop}\label{GRM_GRC}
For any $\LT$-formula $A$, $\GRM \vdash A^\top$ if and only if $\GRC \vdash A^\top$. 
Consequently, $\GRM \vdash A^\top$ if and only if $\GRC \vdash A$. 
\end{prop}
\begin{proof}
We prove only the implication $(\Leftarrow)$ of the first statement. 
Suppose, towards a contradiction, that $\GRC \vdash A^\top$ and $\GRM \nvdash A^\top$. 
Since $\GR \vdash A^\top$, by Fact \ref{compl_GRM}.(2), we have that $A^\top$ is valid in all non-trivial $\GRM$-models. 
So, by Fact \ref{compl_GRM}.(1), there exists a $\GRM$-model $(\{r\}, \emptyset, r, \Vdash)$ such that $r \nVdash A^\top$. 
For each $C \in \MFT$ with $\GRC \nvdash C$, by Theorem \ref{ffp_GRC}, there exist a $\GRC$-model $(W^C, \acc^C, \{\prec_B^C\}_{B \in \MFT}, \Vdash^C)$ and an element $w^C \in W^C$ such that $w^C \nVdash^C C$. 
We define the model $(W^*, \acc^*, \{\prec_B^*\}_{B \in \MFT}, \Vdash^*)$ as follows: 
\begin{itemize}
    \item $W^* = \{r\} \cup \bigcup \{W^C \mid \GRC \nvdash C\}$, 
    \item $\acc^* = \bigcup \{ \acc^C \mid \GRC \nvdash C\}$, 
    \item for each $B \in \MFT$, 
    \[
        \prec_B^* = \begin{cases} \{(r, w^B)\} \cup \bigcup \{ \prec_B^C \mid \GRC \nvdash C\} & \text{if}\ r \nVdash \bb B\ \&\ \GRC \nvdash B, \\
        \bigcup \{ \prec_B^C \mid \GRC \nvdash C\} & \text{otherwise},  
        \end{cases}
    \]
    \item for $x \in W^C$, $x \Vdash^* p \iff x \Vdash^C p$; and $r \Vdash^* p \iff r \Vdash p$. 
\end{itemize}
It is easy to see that the model $(W^*, \acc^*, \{\prec_B^*\}_{B \in \MFT}, \Vdash^*)$ is actually a $\GRC$-model. 
Also, it is shown that $w^C \nVdash^* C$ for each $C \in \MFT$ with $\GRC \nvdash C$. 

\begin{cl*}
    For any $B \in S_0(A^\top)$, we have $r \Vdash^* B$ if and only if $r \Vdash B$. 
\end{cl*}
\begin{proof}
    We prove the claim by induction on the construction of $B \in S_0(A^\top)$. 
    It suffices to show the case that $B$ is of the form $\bb C$. 

    $(\Rightarrow)$: 
    Suppose $r \nVdash \bb C$. 
    Since $\bb C \in S_0(A^\top)$, we have $\GRC \nvdash C$. 
    Then, we obtain $r \prec_C^* w^C$ and $w^C \nVdash^* C$. 
    Thus, $r \nVdash^* \bb C$. 

    $(\Leftarrow)$: Suppose $r \Vdash \bb C$. 
    Then, there is no $x \in W^*$ such that $r \prec_C^* x$. 
    We have $r \Vdash^* \bb C$. 
    \end{proof}
Since $r \nVdash A^\top$, by the claim, $r \nVdash^* A^\top$. 
By Theorem \ref{ffp_GR}, we conclude $\GRC \nvdash A^\top$. 
This is a contradiction. 
\end{proof}

\begin{thm}[UIP for $\GRC$]
$\GRC$ enjoys uniform interpolation property.
\end{thm}
\begin{proof}
    We fix an $\LT$-formula $A$ and a finite set $P$ of propositional variables. 
    By applying UIP for $\GRM$, we find a uniform interpolant $C$ of $(A^\top, P)$ in $\GRM$. 
    That is, 
    \begin{enumerate}
\item $v(C) \subseteq v(A^\top) \setminus P$, 
\item $\GRM \vdash A^\top \to C$, 
\item for any $B \in \MFT$, if $\GRM \vdash A^\top \to B$ and $v(B) \cap P = \emptyset$, then $\GRM \vdash C \to B$.
\end{enumerate}
We prove that $C$ is a uniform interpolant of $(A, P)$ in $\GRC$. 
We have $v(C) \subseteq v(A) \setminus P$ because $v(A^\top) \subseteq v(A)$. 
Since $\GRC \vdash A \leftrightarrow A^\top$, we get $\GRC \vdash A \to C$. 
Let $B$ be any $\LT$-formula such that $\GRC \vdash A \to B$ and $v(B) \cap P = \emptyset$. 
By Proposition \ref{GRM_GRC}, $\GRM \vdash A^\top \to B^\top$. 
Since $v (B^\top) \setminus P \subseteq v(B) \setminus P = \emptyset$, by the choice of $C$, we obtain $\GRM \vdash C \to B^\bot$. 
Then, we conclude $\GRC \vdash C \to B$ because $\GRC \vdash B \leftrightarrow B^\top$. 
\end{proof}

\subsection{UIP for $\GR$}

Finally, we prove UIP for $\GR$. 
Actually, UIP for $\GR$ can be proved by using UIP for $\GRM$ (Theorem \ref{UIP_GRM}) and Sidon's result on the embedding of $\GR$ into $\GRM$ (\cite[Theorem 2.2]{Sidon}). 
Here, we prove a proposition on the embedding of $\GR$ into $\GRC$, and show that UIP for $\GR$ also follows from UIP for $\GRC$ by using our embedding result. 
For each $\LT$-formula $A$, we define the $\LT$-formula $\Theta_A$ by 
\[
 \bigwedge \{\neg \bb C \mid \bb C \in S(A)\ \text{and}\ \GR \vdash \neg C\}. 
\]

\begin{prop}\label{GRC_GR}
For any $\LT$-formula $A$, $\GRC \vdash \Theta_A \to A$ if and only if $\GR \vdash A$. 
\end{prop}
\begin{proof}
$(\Rightarrow)$: Since $\GR \vdash \Theta_A$, this implication trivially holds. 

$(\Leftarrow)$: Suppose, towards a contradiction, that $\GR \vdash A$ and $\GRC \nvdash \Theta_A \to A$. 
By Theorems \ref{ffp_GRC} and \ref{ffp_GR}, there exists a $\GRC$-model $(\{r\}, \emptyset, \{\prec_B\}_{B \in \MFT}, \Vdash)$ such that $r \Vdash \Theta_A$ and $r \nVdash A$. 
For each $C \in \MFT$ with $\GR \nvdash \neg C$, by Theorem \ref{ffp_GR}, we obtain a $\bb$-serial $\GRC$-model $(W^C, \acc^C, \{\prec_B^C\}_{B \in \MFT}, \Vdash^C)$ and an element $w^C \in W^C$ such that $w^C \Vdash^C C$. 
We define the model $(W^*, \acc^*, \{\prec_B^*\}_{B \in \MFT}, \Vdash^*)$ as follows: 
\begin{itemize}
    \item $W^* = \{r\} \cup \bigcup \{W^C \mid \GR \nvdash \neg C\}$, 
    \item $\acc^* = \bigcup \{ \acc^C \mid \GR \nvdash \neg C\}$, 
    \item for each $B \in \MFT$, 
    \[
        \prec_B^* = \begin{cases} \{(r, w^B)\} \cup \bigcup \{ \prec_B^C \mid \GR \nvdash \neg C\}
         & \text{if}\ r \Vdash \bb B\ \&\ \GR \nvdash \neg B, \\
        \{(r, r)\} \cup \bigcup \{ \prec_B^C \mid \GR \nvdash \neg C\} & \text{otherwise},
          \end{cases}
    \]
    \item for $x \in W^C$, $x \Vdash^* p \iff x \Vdash^C p$; and $r \Vdash^* p \iff r \Vdash p$. 
\end{itemize}
It is shown that the model $(W^*, \acc^*, \{\prec_B^*\}_{B \in \MFT}, \Vdash^*)$ is a $\bb$-serial $\GRC$-model. 
Also, it is shown that $w^C \Vdash^* C$ for each $C \in \MFT$ with $\GR \nvdash \neg C$. 

\begin{cl*}
    For any $B \in S(A)$, we have $r \Vdash^* B$ if and only if $r \Vdash B$. 
\end{cl*}
\begin{proof}
    We prove the claim by induction on the construction of $B \in S(A)$. 
    We give only the proof of the case that $B$ is of the form $\bb C$. 

    $(\Rightarrow)$: 
    Suppose $r \nVdash \bb C$. 
    Then, $r \prec_C r$ and $r \nVdash C$. 
    By the induction hypothesis, $r \nVdash^* C$. 
    Since $r \prec_C^* r$, we obtain $r \nVdash^* \bb C$. 

    $(\Leftarrow)$: Suppose $r \Vdash \bb C$. 
    Then, $\neg \bb C$ is not a conjunct of $\Theta_A$, and so $\GR \nvdash \neg C$. 
    Let $x$ be any element of $W'$ such that $r \prec_C^* x$, then $x = w^C$. 
    Since $w^C \Vdash^* C$, we obtain $r \Vdash^* \bb C$. 
    \end{proof}
Therefore, $r \nVdash^* A$. 
By Theorem \ref{ffp_GR}, we have $\GR \nvdash A$, a contradiction. 
\end{proof}

\begin{thm}[UIP for $\GR$]
$\GR$ enjoys uniform interpolation property.
\end{thm}
\begin{proof}
    We fix an $\LT$-formula $A$ and a finite set $P$ of propositional variables. 
    By applying UIP for $\GRC$, we find a uniform interpolant $C$ of $(\Theta_A \land A, P)$ in $\GRC$. 
    That is, 
    \begin{enumerate}
\item $v(C) \subseteq v\bigl(\Theta_A \land A \bigr) \setminus P$, 
\item $\GRC \vdash \Theta_A \land A \to C$, 
\item for any $B \in \MFT$, if $\GRC \vdash \Theta_A \land A \to B$ and $v(B) \cap P = \emptyset$, then $\GRC \vdash C \to B$.
\end{enumerate}
We prove that $C$ is a uniform interpolant of $(A, P)$ in $\GR$. 
Since $v \bigl(\Theta(A) \bigr) \subseteq v(A)$, we have $v(C) \subseteq v(A) \setminus P$. 
Since $\GR \vdash \Theta_A$, we get $\GR \vdash A \to C$. 
Let $B$ be any $\LT$-formula such that $\GR \vdash A \to B$ and $v(B) \cap P = \emptyset$. 
By Proposition \ref{GRC_GR}, $\GRC \vdash \Theta_{A \to B} \to (A \to B)$. 
Then, we have $\GRC \vdash \Theta_A \land A \to \bigl(\Theta_B \to B \bigr)$. 
Since $v \bigl(\Theta_B \to B \bigr) \setminus P = v(B) \setminus P = \emptyset$, by the choice of $C$, we obtain $\GRC \vdash C \to \bigl(\Theta_B \to B \bigr)$. 
Then, we conclude $\GR \vdash C \to B$ because $\GR \vdash \Theta_B$. 
\end{proof}






\section{Future work}\label{sec_fw}

In this paper, we introduced the logic $\GRC$ and its relational semantics. 
By using our semantics, we proved that the logics $\GRC$ and $\GR$ enjoy Lyndon interpolation property. 
Also, from our proofs of these results, we obtained that $\GRC$ and $\GR$ are complete and have the finite frame property with respect to our semantics. 
We also proved that the logics $\GRM$, $\GRC$ and $\GR$ enjoy uniform interpolation property. 
In particular, using Fact \ref{dagger}, we proved that UIP is inherited from $\GL$ to $\GRM$. 
Although $\GL$ has LIP (cf.~Corollary \ref{LIP_GL}), in this case, Fact \ref{dagger} cannot be used to prove LIP for $\GRM$.
This is because in the case of $p \in v^+(A) \setminus v^-(A)$, the variable $p$ may occur both positively and negatively in $\Psi_D$ for $\bb D \in S_0(A)$. 
We propose the following problem. 

\begin{prob}
        Does the logic $\GRM$ enjoy Lyndon interpolation property? 
\end{prob}

In \cite{Kurahashi20}, the second author introduced the notion of uniform Lyndon interpolation property. 

\begin{defn}
We say that a logic $L$ enjoys \textit{uniform Lyndon interpolation property} iff for any formula $A$ and any finite sets $P^+$ and $P^-$ of propositional variables, there exists a formula $C$ satisfying the following conditions:
\begin{enumerate}
    \item $v^*(C) \subseteq v^*(A) \setminus P^*$ for $* \in \{+, -\}$, 
    \item $L \vdash A \to C$, 
    \item for any formula $B$, if $L \vdash A \to B$ and $v^*(B) \cap P^* = \emptyset$ for $* \in \{+, -\}$, then $L \vdash C \to B$. 
\end{enumerate}
\end{defn}

The following problem is a challenging one.

\begin{prob}
        Do the logics $\GRM$, $\GRC$ and $\GR$ enjoy uniform Lyndon interpolation property? 
\end{prob}

As mentioned in the introduction, the second author proved in \cite{Kurahashi} that $\GR$ is a conservative extension of $\NR$. 
Also, we proved in Theorem \ref{conservation} that $\GRC$ is a conservative extension of $\N$. 
However, our theorems established in this paper do not seem to give any interpolation result for $\N$ and $\NR$. 
We suggest the following problem. 

\begin{prob}
    Study interpolation properties for extensions of $\N$. 
\end{prob}

Relating to this problem, the finite frame property of logics obtained from $\N$ by adding the axiom $\bb^n A \to \bb^m A$ is investigated in \cite{KS}.  

\section*{Acknowledgements}

The second author was supported by JSPS KAKENHI Grant Number JP23K03200.

\bibliographystyle{plain}
\bibliography{ref}

\begin{thebibliography}{10}

\bibitem{AB}
Sergei~N. Artemov and Lev~D. Beklemishev.
\newblock Provability logic.
\newblock In D.~Gabbay and F.~Guenthner, editors, {\em Handbook of
  Philosophical Logic}, volume~13, pages 189--360. Springer, Dordrecht, 2nd
  edition, 2005.

\bibitem{Boolos}
George Boolos.
\newblock {\em The unprovability of consistency}.
\newblock Cambridge University Press, Cambridge-New York, 1979.
\newblock An essay in modal logic.

\bibitem{fmt}
{Melvin C.} Fitting, {V. Wiktor} Marek, and Miroslaw Truszczy{\'n}ski.
\newblock The pure logic of necessitation.
\newblock {\em Journal of Logic and Computation}, 2(3):349--373, 1992.

\bibitem{Kurahashi}
Taishi Kurahashi.
\newblock The provability logic of all provability predicates.
\newblock {\em Journal of Logic and Computation}.
\newblock to appear.

\bibitem{Kurahashi20}
Taishi Kurahashi.
\newblock Uniform {L}yndon interpolation property in propositional modal
  logics.
\newblock {\em Archive for Mathematical Logic}, 59(5-6):659--678, 2020.

\bibitem{KS}
Taishi Kurahashi and Yuta Sato.
\newblock The finite frame property of some extensions of the pure logic of
  necessitation.
\newblock 2023.
\newblock arXiv:2305.14762.

\bibitem{Maksimova}
Larisa Maksimova.
\newblock Amalgamation and interpolation in normal modal logic.
\newblock volume~50, pages 457--471. 1991.
\newblock Algebraic logic.

\bibitem{Shamkanov}
D.~S. Shamkanov.
\newblock Interpolation properties of the provability logics {GL} and {GLP}.
\newblock {\em Trudy Matematicheskogo Instituta Imeni V. A. Steklova},
  274:329--342, 2011.

\bibitem{Shavrukov}
V.~Yu. Shavrukov.
\newblock On {R}osser's provability predicate.
\newblock {\em Zeitschrift f\"{u}r Mathematische Logik und Grundlagen der
  Mathematik}, 37(4):317--330, 1991.

\bibitem{Shavrukov93}
V.~Yu. Shavrukov.
\newblock Subalgebras of diagonalizable algebras of theories containing
  arithmetic.
\newblock {\em Dissertationes Mathematicae (Rozprawy Matematyczne)}, 323:82,
  1993.

\bibitem{Sidon}
Tatiana Sidon.
\newblock Craig interpolation property in modal logics with provability
  interpretation.
\newblock In {\em Logical foundations of computer science ({S}t. {P}etersburg,
  1994)}, volume 813 of {\em Lecture Notes in Comput. Sci.}, pages 329--340.
  Springer, Berlin, 1994.

\bibitem{Smorynski}
C.~Smory\'{n}ski.
\newblock Beth's theorem and self-referential sentences.
\newblock In {\em Logic {C}olloquium '77 ({P}roc. {C}onf., {W}roc\l aw, 1977)},
  volume~96 of {\em Stud. Logic Found. Math.}, pages 253--261. North-Holland,
  Amsterdam-New York, 1978.

\bibitem{Solovay}
Robert~M. Solovay.
\newblock Provability interpretations of modal logic.
\newblock {\em Israel Journal of Mathematics}, 25(3-4):287--304, 1976.

\end{thebibliography}

\end{document}